\documentclass[11pt]{article}

\usepackage[margin=1in]{geometry}
\usepackage{amsmath,amssymb,amsthm,mathtools}
\usepackage{thmtools}
\usepackage{bm}
\usepackage{enumitem}

\usepackage[affil-it]{authblk}

\usepackage{hyperref}
\usepackage[nameinlink,noabbrev]{cleveref}
\crefname{appendix}{Appendix}{Appendices}

\hypersetup{
  colorlinks=true,
  linkcolor=blue,
  citecolor=blue,
  urlcolor=blue
}

\newtheorem{theorem}{Theorem}[section]
\newtheorem{lemma}[theorem]{Lemma}
\newtheorem{proposition}[theorem]{Proposition}
\newtheorem{corollary}[theorem]{Corollary}
\newtheorem{remark}[theorem]{Remark}

\newcommand{\C}{\mathbb{C}}
\newcommand{\R}{\mathbb{R}}

\newcommand{\ii}{\mathrm{i}}

\newcommand{\tr}{\operatorname{tr}}

\DeclareMathOperator{\He}{H}
\DeclareMathOperator{\Sk}{S}
\DeclareMathOperator{\Rea}{Re}
\DeclareMathOperator{\Ima}{Im}

\newcommand{\SC}{\mathcal{S}_{\C}}
\newcommand{\SR}{\mathcal{S}_{\R}}
\newcommand{\SF}{\mathcal{S}_{\mF}}

\newcommand{\mR}{\mathbb{R}}

\newcommand{\mC}{\mathbb{C}}

\newcommand{\mF}{{\mathbb F}}
\newcommand{\gln}{{\mathbb{GL}_n}}

\newcommand{\mU}{{\mathbb U}}

\newcommand{\diag}{{\text{diag}}}

\makeatletter
\def\iddots{\mathinner{\mkern1mu\raise\p@
\vbox{\kern7\p@\hbox{.}}\mkern2mu
\raise4\p@\hbox{.}\mkern2mu\raise7\p@\hbox{.}\mkern1mu}}
\makeatother

\usepackage{xcolor}

\title{On the real angular stability radius for sectorial matrices}

\author{Sei Zhen Khong}
\affil{Department of Electrical Engineering\\ National Sun Yat-sen University,\newline Kaohsiung 804201, Taiwan}

\author{Xin Mao}
\affil{School of Data Science and Society\\ University of North Carolina at Chapel Hill,\newline Chapel Hill, NC 27599, USA}

\author{Li Qiu}
\affil{School of Science and Engineering\\ The Chinese
University of Hong Kong,\newline Shenzhen 518172, Guangdong, China}

\author{Axel Ringh%
\thanks{
  Corresponding author; email: \texttt{axelri@chalmers.se}}}
\affil{Department of Mathematical Sciences\\ Chalmers University of Technology and\newline University of Gothenburg\\ 412 96, Gothenburg, Sweden}

\date{\today}

\begin{document}
\maketitle

\begin{abstract}
The real stability radius problem asks for the smallest real perturbation that makes a given matrix lose rank, and it is a structured counterpart of the corresponding complex stability radius problem. In this paper, we consider an angular version of this question for sectorial matrices, where the size of a perturbation is measured by its largest phase instead of its largest singular value. More precisely, for a sectorial matrix $M$, we investigate the minimum required largest phase of a real and sectorial perturbation $\Delta$ for which $I+M\Delta$ loses rank.
We show that for some matrices, the complex and the real angular stability radius are the same, but that they are in general different, such as when $M$ is diagonal or complex symmetric. For these cases, we derive a formula for the real angular stability radius, and this formula depends on the two largest phases of $M$. Moreover, in these cases we show that the result for $n \times n$ matrices can be reduced to that for $2 \times 2$ matrices by compressing an arbitrary real destabilizing perturbation to the real two-dimensional subspace spanned by the real and imaginary parts of a destabilizing vector.
\end{abstract}

\tableofcontents

\section{Introduction}
The problem of computing the real stability radius is a stability radius problem for structured perturbations, as formulated in 
\cite{HinrichsenPritchard1986b}.
This type of problem is of interest in many engineering applications, where it is important that a matrix has all its eigenvalues in a prescribed area of the complex plane.
In particular, such problems are of interest in robust control, where they have been extensively studied \cite{VanLoan1985, HinrichsenPritchard1986a, HinrichsenPritchard1986b,
Byers1988, QiuDavison1989, HinrichsenPritchard1990, QiuDavison1991, QiuDavison1992, QiuBernhardssonRantzerDavisonYoungDoyle1995}; see also the literature on structured singular values, e.g., \cite{Doyle1982, PackardDoyle1993, ZhouDoyleGlover1996, GreenLimebeer1995}.
In their simplest form, a structured stability radius problem can be formulated as follows: given a complex matrix \(M \in \C^{n \times n}\), what is the smallest matrix \(\Delta \in \Omega \subseteq \C^{n\times n}\) such that \(I + M \Delta\) loses rank? That is, solve
\begin{equation}\label{eq:real_stab_radius_problem}
\inf_{\Delta \in \Omega} \; \{ \sigma_1(\Delta) \; \colon  \; \det(I + M \Delta) = 0 \},
\end{equation}
where $\sigma_1(\cdot)$ denotes the largest singular value of a matrix.
When $\Omega = \C^{n \times n}$, it is well-known that the solution is $1/\sigma_1(M)$. The real stability radius problem is where $\Omega = \R^{n \times n}$, and the problem (in fact, a more general problem) was solved in \cite{QiuBernhardssonRantzerDavisonYoungDoyle1995}: for $\C^{n \times n} \ni M = X + iY$, where $X,Y \in \R^{n \times n}$, the reciprocal of the radius is given as the solution to the one-parameter optimization problem
\[
\inf_{\gamma \in (0, 1]} \; \left\{ \sigma_2 \left( \begin{bmatrix}
    X & - \gamma Y \\ \gamma^{-1} Y & X
\end{bmatrix} \right) \right\},
\]
where $\sigma_2(\cdot)$ denotes the second largest singular value.

For a complex scalar $z$ we can write it in polar form $z = re^{i\theta}$, where $r = |z|$ is the magnitude and $\theta = \angle z$ is the phase. The singular values of a matrix are well-established and generally accepted as the concept that generalizes the notion of magnitude to matrices. That is why \eqref{eq:real_stab_radius_problem} is called a stability \emph{radius} problem. However, the notion of phases of a matrix does not have a uniformly accepted definition in the same way. One line of work defines phases for a class of matrices called sectorial matrices \cite{FurtadoJohnson2001, FurtadoJohnson2003, Zhang2015, WangChenKhongQiu2020,WangMaoChenQiu2023}, i.e., phases are defined for matrices whose canonical congruence form is diagonal \cite{HornSergeichuk2006}. This definition of phases is part of a large literature related to sectorial and accretive matrices, as well as the numerical range of a matrix, see, e.g., \cite{HornSteinberg1959, DePrimaJohnson1974, BallantineJohnson1975, JohnsonFurtado2001, ArlinskiiPopov2003, GeorgeIkramov2005, LiSze2014, DruryLin2014, drury2015principal, BedraniKittanehSababheh2021AFA, BedraniKittanehSababheh2021Positivity, zhao2022low, RinghQiu2022, zhang2025matrix}.

\begin{remark}\label{rem:phases_mimo_systems}
The phase theory based on sectorial matrices has also been extended to dynamical systems; for results related to linear time-invariant systems see \cite{ChenWangKhongQiu2019, ChenWangKhongQiu2024, MaoChenQiu2022, yang2025small, zhang2025phantom}, and for results related to nonlinear systems see
\cite{chen2020phase, chen2025singular}. For completeness, we also note that there have been earlier attempts in the literature on dynamical systems to define a phase concept for multiple-input-multiple-output systems, see, e.g.,
\cite{PostlethwaiteEdmundsMacFarlane1981, owens1984numerical, anderson1988hilbert, BarOnJonckheere1990, haddad1992there, chen1998multivariable}.
\end{remark}

In this paper, we consider the phase-analogue of \eqref{eq:real_stab_radius_problem}, namely where the largest singular value, $\sigma_1(\Delta)$, is substituted by the largest phase, $\phi_1(\Delta)$, and where $\Omega$ is taken to be the set of real sectorial matrices. In order for the phases of $M$ to be well-defined, $M$ must also be a sectorial matrix. We defer the exact mathematical formulation of the real angular stability radius to \eqref{eq:strict-radius} in \Cref{subsec:initial}, since some background, notation, and analysis is needed before it can be stated in a mathematically precise way.

Now, let $M$ be sectorial and therefore have well-defined phases. For the case where $\Omega$ is taken to be complex sectorial matrices, the answer to the phase analogue of \eqref{eq:real_stab_radius_problem} was given in \cite{WangChenKhongQiu2020}, and for convenience the result is restated as \Cref{thm:wang2020} in this paper. However, it turns out that for some matrices $M$, the complex and the real angular stability radius are the same (see \Cref{prop:complex_real_margin_same}); this demonstrates a major difference compared to the real stability radius. Nevertheless, when $M$ is a diagonal matrix, or when it is complex symmetric, 
the real angular stability radius and the complex angular stability radius are in general different.
Moreover, the case of complex symmetric matrices $M$ is of special interest in the literature on dynamical systems \cite{yang2025small} (see \Cref{rem:phases_mimo_systems} for context). Therefore, most of the analysis in this paper is concerned with the cases where $M$ is diagonal or complex symmetric.

The outline of the paper is as follows: in \Cref{sec:background} we give relevant background. In particular, we define the notion of sectorial matrices and phases of such matrices, and we present the exact mathematical formulation of the real angular stability radius problem that we consider. Moreover, we also show that there are sectorial matrices $M$ for which $\theta_\C(M) = \theta_\R(M)$ (see \Cref{prop:complex_real_margin_same}). In \Cref{sec:main_results}, which contains the main results of the paper, we restrict our attention to diagonal $M$ (\Cref{thm:main}), and $M$ that is complex symmetric (\Cref{cor:complex_symmetric}),
in which cases we in general have $\theta_\C(M) < \theta_\R(M)$. The proof of \Cref{thm:main} is presented in \Cref{sec:proof_main_result}. Some conclusions and outlooks are presented in \Cref{sec:conclusion}.
Finally, in \Cref{app:2times2}, the result of \Cref{thm:main} is established for the case of $2 \times 2$ matrices. We use a different technique in the proof compared to that for $n \times n$ matrices and present it for the interested reader.

\section{Background, problem formulation, and initial analysis}\label{sec:background}

In this section we present the relevant background material needed, and develop the exact mathematical formulation of the real angular stability radius problem considered in this paper. We also show that there are matrices for which the complex and the real angular stability radii are the same. First, we set up the notation, some of which has already been implicitly used in the introduction.

\paragraph{Notation}
Let $\mR$ and $\mC$ denote the field of real numbers and the field of complex numbers, respectively. Since we work with both fields, we use the notation  $\mF$ to denote either of the two fields.
To this end, for $A \in \mF^{n \times m}$ let $A^T$ denote its transpose and $A^*$ denote its complex conjugate transpose. The direct sum of two matrices is denoted by $\oplus$.
For a vector $x \in \mF^{n}$, $\|x\| := \sqrt{x^*x}$ denotes the Euclidean norm.
$A \in \mF^{n \times n}$ is called Hermitian if $A^* = A$, and skew-Hermitian if $A^* = -A$. A matrix $A \in \R^{n \times n}$ is called symmetric if $A^T = A$, and $A \in \C^{n \times n}$ is called complex symmetric if the same equality holds. Next, 
$\He(A) := \tfrac{1}{2} (A + A^*)$ denotes the Hermitian part of a matrix, and $\Sk(A) := \tfrac{1}{2} (A - A^*)$ denotes its skew-Hermitian part. Moreover, for $A \in \mF^{n \times n}$, let $\lambda(A)$ denote its spectrum.
Let $\gln(\mF)$ denote the set of invertible matrices of size $n \times n$, and let $\mU_n(\mF)$ denote the set of unitary matrices of size $n \times n$. Note that for $\mF = \mR$, the latter is the set of orthogonal matrices. For Hermitian matrices, we use $\succ 0$ ($\succeq 0$) to denote that they are positive (semi-) definite. A matrix $A$ such that $\He(A) \succ 0$ is called strictly accretive. Finally, two matrices $A,B \in \mF^{n \times n}$ are said to be congruent if there exists a matrix $C \in \gln(\mC)$ such that $A = C^*BC$.

\subsection{Phases of a sectorial matrix}
Given a matrix $A \in \mF^{n \times n}$, we denote its numerical range by
\[
W(A) := \big\{ z \in \mC \mid z = x^*Ax, \; x \in \mC^n \text{ and } \| x \|^2 = x^*x = 1 \big\},
\]
which by the Toeplitz–Hausdorff theorem is a compact convex set in $\C$ for all matrices $A$ \cite[Prop.~1.2.1 and 1.2.2]{horn1994topics}. We denote the angular field of values by
\[
W'(A) := \big\{ z \in \mC \mid z = x^*Ax, \; x \in \mC^n \setminus \{0\} \big\},
\]
which is the smallest convex cone that contains the numerical range.
$A \in \mF^{n \times n}$ is said to be sectorial if $0 \not \in W(A)$. Note that this can only hold for nonsingular matrices $A$ since otherwise $0 \in \lambda(A) \subset W(A)$; see \cite[Property~1.2.6]{horn1994topics}. For sectorial matrices, we can define the concept of phases via the following construction (see \cite{WangChenKhongQiu2020}):
if $A$ is sectorial, it is congruent to a unitary diagonal matrix $D$, i.e.,
\begin{equation}\label{eq:sect_decomp}
        A=T^*\diag(e^{\ii\phi_1(A)},\ldots,e^{\ii\phi_n(A)})T,
        \qquad T\in \gln(\mC),    
\end{equation}
where, after choosing an interval of length strictly less than \(\pi\), the phases
\(\phi_1(A)\ge\cdots\ge\phi_n(A)\) are uniquely determined
\cite{FurtadoJohnson2001, HornSergeichuk2006,Zhang2015, WangChenKhongQiu2020}. We use the notation $\phi(A) = (\phi_1(A), \ldots, \phi_n(A))$ to denote a vector of all phases. The decomposition in \eqref{eq:sect_decomp} is called a sectorial decomposition \cite{Zhang2015}.
In particular, for a sectorial matrix $A$ we have that the largest and smallest phases are given by \cite[Lemma~8]{HornSteinberg1959}
\begin{align*}
& \phi_1(A) := \max_{\substack{x \in \C^n \\ \| x \| = 1}} \; \angle x^*Ax, \\
& \phi_n(A) := \min_{\substack{x \in \C^n \\ \| x \| = 1}} \; \angle x^*Ax,
\end{align*}
respectively, where $\angle z$ denotes the phase of the complex number $z$. These are angles of the two supporting rays of the numerical range.
Using the smallest and largest phases, we define the set of phase-constrained matrices as
\[
\SF[\alpha, \beta] := \big\{ A \in \mF^{n \times n} \mid A \text{ sectorial}, \; \phi_1(A) \leq \alpha, \;  \phi_n(A) \geq \beta \big\}.
\]
For the special case $\beta = - \alpha$, we simply write $\SF[\alpha]$. Analogously, we write $\SF(\alpha, \beta)$ (and $\SF(\alpha)$) for the set of all sectorial matrices $A$ where $\phi_1(A) < \alpha$ and $\phi_n(A) > \beta$.

When it comes to angular stability radii, note that one could potentially rotate both clockwise and counter-clockwise in order to cause destabilization. Hence, both the largest and the smallest phases of a matrix might be of interest, and we therefore define the complex angular stability radius as
\begin{equation}\label{eq:complex-radius}
 \theta_{\C}(M):=
 \inf\{ \max\{|\phi_1(\Delta)|,|\phi_n(\Delta)|\}  \; : \; \Delta \in \C^{n \times n},\ \Delta \text{ is sectorial},\ \det(I+M\Delta)=0\}.
\end{equation}
The following theorem gives the solution to the problem of computing the complex angular stability radius.

\begin{theorem}[{\cite[Thm.~8.1]{WangChenKhongQiu2020}}]\label{thm:wang2020}
Let $M\in\C^{n\times n}$ be sectorial with phases in $(-\pi,\pi)$.  Then
\[
        \det(I+M\Delta)\ne0
        \quad\text{for all }\Delta\in\SC[\alpha,\beta]
\]
if and only if
\[
        \phi_1(M)<\pi-\alpha,
        \qquad
        \phi_n(M)>-\pi-\beta .
\]
\end{theorem}

\begin{remark}\label{rem:open_set_non-strict_inequ}
    The result in \Cref{thm:wang2020} is stated in terms of the closed set $\SC[\alpha,\beta]$ and strict inequalities for $\phi_1(M)$ and $\phi_n(M)$. By using \cite[Thm.~6.2]{WangChenKhongQiu2020}, this can be easily changed to the open set $\SC(\alpha,\beta)$ and non-strict inequalities for $\phi_1(M)$ and $\phi_n(M)$.
\end{remark}

\begin{remark}
    This subsection defines phases for sectorial matrices. It can be extended to define phases for so-called quasi- and semi-sectorial matrices based on the canonical congruence form of matrices \cite{HornSergeichuk2006}. Phases for such matrices are not used here, and for simplicity these definitions are thus omitted. Nevertheless, we refer the interested reader to \cite{WangMaoChenQiu2023} or \cite[Sec.~3]{zhang2025matrix}.
\end{remark}

\subsection{Problem formulation and initial analysis}\label{subsec:initial}

For the real angular stability radius, we are interested in perturbations in the set
$\SR[\alpha, \beta]$. However, since the numerical range of a real matrix is symmetric around the real axis \cite[p.~24]{horn1994topics}, for $\Delta \in \mR^{n \times n}$, $\Delta$ is sectorial if and only if either $\Delta \in \SR(\pi/2)$, or $\Delta \in \SR(3\pi/2, \pi/2)$. The former is equivalent to $\Delta$ being strictly accretive, i.e., to $\He(\Delta) \succ 0$, and the latter is equivalent to $\He(\Delta) \prec 0$. Without loss of generality, we therefore focus on the case $\He(\Delta) \succ 0$; analysis of the case $\He(\Delta) \prec 0$ can be done similarly and is omitted.
Next, note that for $n=1$, the problem is trivial: a real strictly accretive scalar is strictly positive, and without loss of generality we can write $M = e^{\ii\mu_1}$.  Now,
\[
        1+e^{\ii\mu_1}\Delta\ne0
\]
for every positive real $\Delta$, unless $\mu_1 = -\pi$ in which case $\Delta = 1$ is destabilizing. For the rest of this paper, we therefore focus on the case $n \geq 2$.

Finally, since by construction we have that $\theta_{\R}(M) \ge \theta_{\C}(M)$, and since the set of perturbations is $\SR(\pi/2)$, by \Cref{thm:wang2020} and \Cref{rem:open_set_non-strict_inequ} we have that $\theta_{\R}(M) = \infty$ (i.e., there is no strictly accretive real $\Delta$ such that $I + M\Delta$ loses rank) if $\phi_1(M) \leq \pi/2$ and $\phi_n(M) \geq -\pi/2$. Thus, the nontrivial cases are those in which a chosen sectorial phase branch crosses one of the two imaginary axes. In the main theorem, \Cref{thm:main} we treat the upper-dominant diagonal case, where the phases can be ordered as
\[
        \pi>\mu_1\ge\cdots\ge\mu_n>-\pi/2,
        \qquad \mu_1-\mu_n<\pi.
\]
This includes mixed-sign phase configurations. The lower-dominant case follows by conjugation and is stated separately in \Cref{cor:lower}.

Taking all of this together, this means that the real angular stability radius considered in this work can be defined as
\begin{equation}\label{eq:strict-radius}
 \theta_{\R}(M):=
 \inf\{\phi_1(\Delta) \; : \; \Delta\in\R^{n\times n},\ \He(\Delta) \succ 0,\ \det(I+M\Delta)=0\}.
\end{equation}

In this setup, it turns out that there exist matrices for which the complex and the real angular stability radii are the same.

\begin{proposition}[Complex and real angular stability radius can be the same]
\label{prop:complex_real_margin_same}
Let \(n\ge 2\), and let
\[
        \pi>\mu_1\ge\cdots\ge\mu_n>-\pi/2,
        \qquad \mu_1-\mu_n<\pi.
\]
If \(\mu_1>\pi/2\), then there exists a sectorial matrix
\(M_\mu \in \C^{n\times n}\) with phases $\phi(M_\mu)=(\mu_1,\ldots,\mu_n)$
such that
\[
        \theta_{\R}(M_\mu) = \pi-\mu_1 = \theta_{\C}(M_\mu).
\]
\end{proposition}

\begin{proof}
The fact that $\theta_{\C}(M_\mu) = \pi-\mu_1$ follows from \Cref{thm:wang2020}. The fact that there exists an $M$ such that $\theta_{\R}(M_\mu) = \pi-\mu_1$ is proved by an explicit construction of such an $M_\mu$. To this end, set
\[
        D_\mu:=\operatorname{diag}(e^{\ii\mu_1},\ldots,e^{\ii\mu_n})
\]
and define
\[
        U_2:=
        \frac{1}{\sqrt 2}
        \begin{bmatrix}
        1&1\\
        -\ii&\ii
        \end{bmatrix},
        \qquad \text{and} \qquad
        W:=U_2\oplus I_{n-2}.
\]
Let
\[
        M_\mu:=W D_\mu W^*
\]
and note that since \(M_\mu=(W^*)^*D_\mu W^*\) is a sectorial decomposition of \(M_\mu\), \(M_\mu\) is sectorial and \(\phi(M_{\mu}) = (\mu_1,\ldots,\mu_n) \).
Now, set
\[
        \theta_0:=\pi-\mu_1 .
\]
Because \(\mu_1>\pi/2\), we have \(0<\theta_0<\pi/2\). Define the real rotation
\[
        Q_{\theta_0}:=
        \begin{bmatrix}
        \cos\theta_0&-\sin\theta_0\\
        \sin\theta_0& \cos\theta_0
        \end{bmatrix},
        \qquad
        \Delta_{\theta_0}:=Q_{\theta_0}\oplus I_{n-2}.
\]
Then
\[
        \tfrac{1}{2}(\Delta_{\theta_0}+\Delta_{\theta_0}^T)
        =
        \cos\theta_0 I_2 \oplus I_{n-2} \succ 0,
\]
so \(\Delta_{\theta_0}\) is real and strictly accretive. Moreover, $\phi(\Delta_{\theta_0}) = (\theta_0, 0,\ldots,0, -\theta_0)$, and hence
\[
        \phi_1(\Delta_{\theta_0})=\theta_0=\pi-\mu_1.
\]
Next, a direct computation gives
\[
        U_2^*Q_{\theta_0}U_2
        =
        \operatorname{diag}(e^{\ii\theta_0},e^{-\ii\theta_0}),
\]
and therefore
\[
        W^*\Delta_{\theta_0}W
        =
        \operatorname{diag}(e^{\ii\theta_0},e^{-\ii\theta_0},1,\ldots,1).
\]
Consequently,
\[
\begin{aligned}
        \det(I+M_\mu\Delta_{\theta_0})
        &=
        \det\!\left(I+WD_\mu W^*\Delta_{\theta_0}\right) \\
        &= 
        \det\!\left(I+D_\mu W^*\Delta_{\theta_0}W\right)\\
        &=
        \left(1+e^{\ii(\mu_1+\theta_0)}\right)
        \left(1+e^{\ii(\mu_2-\theta_0)}\right)
        \prod_{j=3}^n(1+e^{\ii\mu_j}) .
\end{aligned}
\]
Since \(\mu_1+\theta_0=\pi\), the first factor is zero, and hence
$\det(I+M_\mu\Delta_{\theta_0})=0$.
\end{proof}

\section{Main results}\label{sec:main_results}

From \Cref{prop:complex_real_margin_same} we know that there are sectorial matrices for which the complex and the real angular stability radii are the same. However, this is not true for all sectorial matrices $M$. In particular, it is not true for the case when $M$ is sectorial and diagonal; that is the conclusion from \Cref{thm:main} and from \Cref{prop:for_diag_complex_real_not_same}. Neither is it true for sectorial and complex symmetric $M$ (see \Cref{cor:complex_symmetric}).

\begin{theorem}[Real angular stability radius for diagonal matrices]\label{thm:main}
Let $n \geq 2$ and let
\[
        M=\diag(e^{\ii\mu_1},\ldots,e^{\ii\mu_n}),
        \qquad
        \pi>\mu_1\ge \mu_2\ge\cdots\ge\mu_n>-\pi/2,
        \qquad
        \mu_1-\mu_n<\pi.
\]
It holds that:
\begin{enumerate}[label=(\roman*)]
\item If \(\mu_1+\mu_2\le\pi\), then there is no
\(\Delta\in\R^{n\times n}\) with \(\Delta+\Delta^T \succ 0\) and
\(\det(I+M\Delta)=0\).  Thus \(\theta_{\R}(M)=+\infty\).

\item If \(\mu_1+\mu_2>\pi\), then
\begin{equation}\label{eq:main-formula}
        \theta_{\R}(M)=
        \arccos\!\left(
        -\frac{\cos((\mu_1+\mu_2)/2)}
              {\cos((\mu_1-\mu_2)/2)}
        \right).
\end{equation}
\end{enumerate}
\end{theorem}

\begin{proof}
See \Cref{sec:proof_main_result}.
\end{proof}

The result can, in a formulation that is more similar to the one in \Cref{thm:wang2020}, also be expressed as follows.
\begin{corollary}\label{cor:main-open-bound}
With the same setup as in \Cref{thm:main}, for \(0\le\alpha<\pi/2\),
\[
        \det(I+M\Delta)\ne0
        \quad\text{for all }\Delta\in\SR[\alpha]
\]
if and only if
\[
        \alpha<
        \begin{cases}
        \displaystyle
        \arccos\!\left(
        -\frac{\cos((\mu_1+\mu_2)/2)}
              {\cos((\mu_1-\mu_2)/2)}
        \right), & \mu_1+\mu_2>\pi,\\[2ex]
        \pi/2, & \mu_1+\mu_2\le\pi.
        \end{cases}
\]
\end{corollary}

\begin{remark}[Boundary case]\label{rem:boundary}
If \(\mu_1+\mu_2=\pi\), the formula in \eqref{eq:main-formula} gives the limiting value \(\pi/2\). However, the largest phase of $\Delta$ cannot be $\pi/2$ under the strict accretivity constraint $\He(\Delta)\succ0$.
\end{remark}

However, from \Cref{prop:complex_real_margin_same} we know that if we do not restrict ourselves to diagonal $M$, then there are matrices for which the complex and the real angular stability radius are the same. Nevertheless, for diagonal matrices, the complex angular stability radius is, in general, strictly smaller than the real angular stability radius. This can be made precise as follows.

\begin{proposition}[Different complex and real radii for diagonal matrices]
\label{prop:for_diag_complex_real_not_same}
Let \(n\ge 2\), let
\[
        \pi>\mu_1 \ge \mu_2 \ge \cdots \ge \mu_n>-\pi/2,
        \qquad
        \mu_1-\mu_n<\pi,
\]
and let $M = \diag(e^{\ii\mu_1},\ldots,e^{\ii\mu_n})$.
If \(\mu_1>\pi/2\), then
\[
        \theta_{\C}(M)=\pi-\mu_1 \leq \theta_{\R}(M),
\]
 where $\theta_{\R}(M)$ is given by \eqref{eq:main-formula} if $\mu_1 + \mu_2  > \pi$, and where $\theta_{\R}(M) = \infty$ otherwise. Moreover, the inequality is strict if $\mu_1 > \mu_2$.
\end{proposition}

\begin{proof}
Since $\mu_n>-\pi/2$ and $\mu_1>\pi/2$, \Cref{thm:wang2020} gives $\theta_{\C}(M)=\pi-\mu_1$. Let
\[
        \theta_\C:=\pi-\mu_1 \quad \text{ and } \quad \theta_\R=
        \arccos\!\left(
        -\frac{\cos((\mu_1+\mu_2)/2)}
              {\cos((\mu_1-\mu_2)/2)}
        \right).
\]
The only thing that does not directly follow from
\Cref{thm:wang2020} and \Cref{thm:main} is the strict inequality $\theta_{\C} < \theta_{\R}$ in the case $\mu_1 + \mu_2 > \pi$.
To show this, note that then $\mu_2>0$ and both angles lie in
\((0,\pi)\). Therefore, it suffices to compare cosines:
\[
\begin{aligned}
        \cos\theta_\C-\cos\theta_\R
        &=-\cos\mu_1
          +\frac{\cos((\mu_1+\mu_2)/2)}
                 {\cos((\mu_1-\mu_2)/2)} \\
        &=\frac{\sin((\mu_1-\mu_2)/2)\sin\mu_1}
                {\cos((\mu_1-\mu_2)/2)}
        >0
\end{aligned}
\]
whenever $\mu_1>\mu_2$. Thus, $\theta_{\C}(M)=\pi-\mu_1 < \theta_{\R}(M)$ whenever $\mu_1 + \mu_2 > \pi$.
\end{proof}

Finally, from \Cref{thm:main}, we can easily derive the following corollaries.

\begin{corollary}[Real congruence invariance]\label{cor:congruence}
Let \(M\in\C^{n\times n}\) be sectorial, let \(R\in \gln(\mR)\), and put \(\widetilde M=R^TMR\). Then
\(\widetilde M\) and \(M\) have the same real angular stability radius.
\end{corollary}

\begin{proof}
Since $R$ is invertible and the determinant is a multiplicative map, we have that
\[
        \det(I+R^TMR\Delta)=\det(R^{T})\det(I+MR\Delta R^T)\det(R^{-T}) = \det(I+M(R\Delta R^T)).
\]
The map \(\Delta\mapsto R\Delta R^T\) is a bijection on real matrices with positive definite symmetric part. Moreover, matrix phases are invariant under congruence, so
\(\phi_1(R\Delta R^T)=\phi_1(\Delta)\).  The destabilizing sets therefore correspond exactly.
\end{proof}

\begin{corollary}[Real angular stability radius for complex symmetric matrices]\label{cor:complex_symmetric}
Let \(n\ge 2\) and $M \in \C^{n \times n}$ be complex symmetric and sectorial with phases $\phi(M)=(\mu_1,\ldots,\mu_n)$, where
\[
        \pi>\mu_1 \ge \mu_2 \ge \cdots \ge \mu_n>-\pi/2,
        \qquad
        \mu_1-\mu_n<\pi.
\]
Then the conclusion is the same as in \Cref{thm:main}.
\end{corollary}

\begin{proof}
A matrix $M \in \C^{n \times n}$ is complex symmetric and sectorial if and only if the congruence transformation in a sectorial decomposition of $M$ is by real matrices, i.e., for a sectorial matrix $M$ with sectorial decomposition $M = T^* \diag(e^{\ii \phi_1(M)}, \ldots, e^{\ii \phi_n(M)}) T$, $M$ is complex symmetric if and only if $T \in \gln(\mR)$ \cite[Thm.~1]{yang2025small}. The result now follows from \Cref{cor:congruence} and \Cref{thm:main}.
\end{proof}

\begin{corollary}[Lower-dominant version]\label{cor:lower}
Let $M=\diag(e^{\ii\mu_1},\ldots,e^{\ii\mu_n})$, where
\[
        \pi/2>\mu_1\ge\cdots\ge\mu_n>-\pi,
        \qquad
        \mu_1-\mu_n<\pi.
\]
It holds that:
\begin{enumerate}[label=(\roman*)]
\item If \(\mu_{n-1}+\mu_n\ge -\pi\), then there is no
\(\Delta\in\R^{n\times n}\) with \(\Delta+\Delta^T \succ 0\) and
\(\det(I+M\Delta)=0\).  Thus \(\theta_{\R}(M)=+\infty\).

\item If \(\mu_{n-1}+\mu_n< -\pi\), then
\[
        \theta_{\R}(M)=
        \arccos\!\left(
        -\frac{\cos((\mu_{n-1}+\mu_n)/2)}
              {\cos((\mu_{n-1}-\mu_n)/2)}
        \right).
\]
\end{enumerate}
\end{corollary}

\begin{proof}
For real \(\Delta\),
\[
        \det(I+M\Delta)=0
        \quad\Longleftrightarrow\quad
        \det(I+\overline M\Delta)=0.
\]
The perturbation class and \(\phi_1(\Delta)\) are unchanged. The phases of $\overline M$ can be ordered as
\[
        -\mu_n\ge -\mu_{n-1}\ge\cdots\ge -\mu_1,
\]
and these phases satisfy the assumptions of \Cref{thm:main}. The result therefore follows from \Cref{thm:main} applied to $\overline M$.
\end{proof}

\section{Proof of Theorem~\ref{thm:main}}\label{sec:proof_main_result}
In this section, the proof of \Cref{thm:main} is presented. The main idea  of the proof is as follows: if $\det(I+M\Delta)=0$, we can consider a nonzero eigenvector $z$ corresponding to the zero eigenvalue. With this eigenvector, we use $\Rea z$ and $\Ima z$ to span a plane in $\R^n$, and then compress $\Delta$ to this plane. This gives a $2 \times 2$ accretive matrix whose phases can be computed, and from which we obtain a lower bound on the phases needed for destabilization. Moreover, we also get a necessary condition for when this can hold, i.e., a witness; we therefore call the plane spanned by $\Rea z$ and $\Ima z$ the witness plane. We then show that this lower bound is related to the largest and second largest phases of $M$, and finally show that the bound is in fact tight by constructing a destabilizing $\Delta$ that attains it.

Throughout this section, we define
\[
        J:=\begin{bmatrix}0&-1\\ 1&0\end{bmatrix}.
\]

\begin{remark}
The underlying idea of the proof is a phase-theoretic adaptation of the real-vector lifting used in the proof of the real stability radius \cite{QiuBernhardssonRantzerDavisonYoungDoyle1995}. The result in \cite{QiuBernhardssonRantzerDavisonYoungDoyle1995} is proved by the same split described in preceding paragraph, into real and imaginary part, and the worst-case real perturbation then constructed turns out to be of rank at most two. The proof presented here uses a similar mechanism, but replaces the singular-value analysis by a compression argument for matrix phases.
\end{remark}

\subsection{Auxiliary results}\label{sec:auxiliary-identities}

This subsection contains a number of auxiliary results that we need in the proof.

\begin{lemma}[Phase of a real \(2\times2\) strictly accretive matrix]\label{lem:2by2-phase}
Let \(B\in\R^{2\times2}\) and suppose \(H:=\He(B) \succ 0\).  Write
\[
        \Sk(B)=sJ,\qquad s\in\R.
\]
Then the two phases of \(B\) are
\[
        \pm\arctan \frac{|s|}{\sqrt{\det H}}.
\]
Equivalently,
\begin{equation}\label{eq:2by2-phase-formula}
        \tan\phi_1(B)=\frac{|s|}{\sqrt{\det H}} \quad \text{and} \quad \tan\phi_2(B)=-\frac{|s|}{\sqrt{\det H}}.
\end{equation}
\end{lemma}

\begin{proof}
The proof has two steps.  First remove the positive definite symmetric part by congruence; then compute the eigenvalues of the resulting real normal form.

By congruence invariance of matrix phases, \(B\) has the same phases as
\[
        H^{-1/2}BH^{-1/2}=I+sH^{-1/2}JH^{-1/2}.
\]
For every positive definite \(2\times2\) matrix \(H\),
\[
        H^{-1/2}JH^{-1/2}
\]
has eigenvalues \(\pm \ii/\sqrt{\det H}\).  Hence the eigenvalues of the congruent sectorial factor are
\[
        1\pm \ii\frac{s}{\sqrt{\det H}},
\]
and their phases are \(\pm\arctan(|s|/\sqrt{\det H})\).
\end{proof}

\begin{lemma}[Phase monotonicity under real compression]\label{lem:compression}
Let \(\Delta\in\R^{n\times n}\) be sectorial and \(X\in\R^{n\times k}\) have full column rank.  Then \(X^T\Delta X\) is sectorial and
\begin{equation}\label{eq:compression-monotonicity}
        \phi_1(X^T\Delta X)\le \phi_1(\Delta).
\end{equation}
If \(\Delta+\Delta^T \succ 0\), then \(X^T\Delta X\) is strictly accretive.
\end{lemma}

\begin{proof}
Write the QR factorization \(X=UR\), where \(U^TU=I_k\) and \(R\in \mathbb{GL}_k(\R)\).  Since phases are invariant under congruence,
\[
        \phi_1(X^T\Delta X)=\phi_1(U^T\Delta U).
\]
The compression interlacing theorem for matrix phases gives
\(\phi_1(U^T\Delta U)\le\phi_1(\Delta)\); see
\cite[Lemma~4.1 and Remark~4.2]{WangChenKhongQiu2020}.  If \(\Delta+\Delta^T \succ 0\), then
\[
        \He(X^T\Delta X)=X^T\He(\Delta)X \succ 0,
\]
so the compression is strictly accretive.
\end{proof}

\begin{lemma}[Polygon lower bound]\label{lem:polygon}
For $j = 1, \ldots, n$,  let  \(\eta_j\in\R\), \(p_j\ge0\), and let \(\sum_{j=1}^n p_j=1\).  Put
\[
        \rho:=\max\{0,2\max_jp_j-1\}.
\]
Then
\begin{equation}\label{eq:polygon-bound}
        \left|\sum_{j=1}^n p_je^{\ii\eta_j}\right|\ge \rho.
\end{equation}
\end{lemma}

\begin{proof}
Let \(\hat{p} := \max_jp_j \).  If \(\hat{p} \le1/2\), the right-hand side of \eqref{eq:polygon-bound} is zero and the inequality therefore holds trivially.  If \(\hat{p} > 1/2\), then $\rho = 2\hat{p}-1 > 0$ and there is also a unique index \( \hat{j} \) with \(p_{\hat{j}}=\hat{p}\). In this case, the expression in \eqref{eq:polygon-bound} is proved by the following chain of inequalities:
\[
        \left|\sum_{j=1}^n p_je^{\ii\eta_j}\right|
        = \left|\hat{p} e^{\ii\eta_{\hat{j}}} - \sum_{j \ne \hat{j}} p_je^{\ii(\eta_j + \pi)}\right|
        \geq \left| \left|\hat{p} e^{\ii\eta_{\hat{j}}}\right| - \left|\sum_{j \ne \hat{j}} p_je^{\ii(\eta_j + \pi)} \right| \right|
        \ge \hat{p} - \sum_{j \ne \hat{j}}p_j
        =2\hat{p}-1 = \rho.
\]
where, in particular, the first inequality is the reverse triangle inequality, and the second inequality uses the triangle inequality on the second term as well as the fact that $\hat{p} > 1/2 > \sum_{j \ne \hat{j}}p_j$.
\end{proof}

\subsection{A bound based on two-dimensional restriction onto the witness plane}\label{sec:forced-compression}

This section converts an arbitrary real destabilizing perturbation into a two-dimensional real witness. The main idea is as follows. If $\Delta$ is destabilizing and \((I+M\Delta)z=0\), then the real and imaginary parts of \(z\) span a real plane on which the restriction of \(\Delta\) is shown to be strictly accretive and already carry a definite phase; we call this plane the witness plane.
Once the problem has been compressed to that plane, the largest phase of a real \(2\times2\) strictly accretive matrix can be computed, and this phase gives a lower bound for the largest phase of the destabilizing $\Delta$. Subsequently, we also get a witness, i.e., a necessary condition, for when a destabilizing $\Delta$ exists.

\begin{lemma}[Independence of vectors spanning the witness plane]
\label{lem:witness-plane-independent}
Let
\[
        M=\diag(e^{\ii\mu_1},\ldots,e^{\ii\mu_n}),
        \qquad -\pi<\mu_j<\pi,
\]
and \(\Delta\in\R^{n\times n}\) satisfy \(\Delta+\Delta^T \succ 0\).  If
\(z\in\C^n\setminus\{0\}\) satisfies
\begin{equation}\label{eq:witness-equation}
        \Delta z=-M^{-1}z,
\end{equation}
then \(\Rea z\) and \(\Ima z\) are linearly independent over \(\R\).
\end{lemma}

\begin{proof}
Suppose that \(\Rea z\) and \(\Ima z\) are linearly dependent.  Then
\(z=cx\) for some \(c\in\C\setminus\{0\}\) and some
\(x\in\R^n\setminus\{0\}\).  Dividing \eqref{eq:witness-equation} by \(c\) gives
\[
        \Delta x=-M^{-1}x .
\]
The left-hand side is real.  Hence each coordinate
\(-e^{-\ii\mu_j}x_j\) is real.  Since $-\pi<\mu_j<\pi$, this can happen with $x_j\ne0$ only when $\mu_j=0$. Therefore $x$ is supported only on the zero-phase coordinates. Since $M$ is diagonal, on that support of $x$ the elements on the diagonal are equal to $1$ and hence the corresponding elements in $M^{-1}$ are also equal to $1$. Therefore we must have that
\[
        \Delta x=-x .
\]
Multiplying from the left by $x^T$ gives
\[
        x^T\Delta x=-\|x\|^2<0.
\]
But \(\Delta+\Delta^T \succ 0\) implies
\[
        x^T\Delta x=x^T\He(\Delta)x>0,
\]
which is a contradiction.  Thus \(\Rea z\) and \(\Ima z\) are linearly independent.
\end{proof}

The next lemma contains a number of identities that, as the name of the lemma indicates, have to do with the witness plane.

\begin{lemma}[Coordinate identities for the witness plane]\label{lem:coordinate-identities}
Let \(z\in\C^n\setminus\{0\}\), let $-\pi<\mu_j<\pi$, for $j = 1, \ldots, n$, and define
\begin{equation}\label{eq:p-q-r-rho-def}
        p_j:=\frac{|z_j|^2}{\|z\|^2},
        \qquad
        q:=-\sum_{j=1}^n p_j\cos\mu_j,
        \qquad
        r:=\sum_{j=1}^n p_j\sin\mu_j,
        \qquad
        \rho:=\max\{0,2\max_jp_j-1\}.
\end{equation}
Next, set
$w_j :=-e^{-\ii\mu_j}z_j$,
\[
        X:=\begin{bmatrix}\Rea z&\Ima z\end{bmatrix},
        \quad
        Y:=\begin{bmatrix}\Rea w&\Ima w\end{bmatrix},
        \quad
        A:=X^TY \in \R^{2 \times 2},
        \quad
        H:=\He(A),
        \quad
        \Sk(A)=sJ.
\]
Then
\begin{equation}\label{eq:trace-skew-identities}
        \tr H=\|z\|^2q,
        \qquad
        |s|=\frac{\|z\|^2}{2}|r|.
\end{equation}
Moreover, writing the traceless part of \(H\) as
\[
        H-\frac{\tr H}{2}I_2=
        \begin{bmatrix}a&b\\ b&-a\end{bmatrix},
\]
there exist angles \(\eta_j\) such that
\begin{equation}\label{eq:traceless-vector-identity}
        a+\ii b=\frac{\|z\|^2}{2}\sum_{j=1}^n p_je^{\ii\eta_j}.
\end{equation}
Consequently,
\begin{equation}\label{eq:traceless-lower-det-upper}
        a^2+b^2\ge \frac{\|z\|^4}{4}\rho^2
        \qquad \text{and} \qquad
        \det H\le \frac{\|z\|^4}{4}(q^2-\rho^2).
\end{equation}
\end{lemma}

\begin{proof}
First, write $z$ in polar form as \(z_j=|z_j|e^{\ii\psi_j}\), for $j = 1, \ldots, n$. Since $w_j=e^{\ii(\pi-\mu_j)}z_j$, the \(j\)-th rows of \(X\) and \(Y\) are
\[
        x_j=|z_j|\begin{bmatrix}\cos\psi_j&\sin\psi_j\end{bmatrix},
        \qquad
        y_j=|z_j|\begin{bmatrix}\cos(\psi_j+\pi-\mu_j)&\sin(\psi_j+\pi-\mu_j)\end{bmatrix}.
\]
The contribution of the \(j\)-th coordinate to \(A=X^TY\) is 
\[
x_j^Ty_j = |z_j|^2
\begin{bmatrix}
\cos\psi_j \cos(\psi_j+\pi-\mu_j) & \cos\psi_j \sin(\psi_j+\pi-\mu_j) \\ \sin\psi_j\cos(\psi_j+\pi-\mu_j) & \sin\psi_j\sin(\psi_j+\pi-\mu_j)
\end{bmatrix}
\]
A direct calculation, using the difference-of-angle formula for $\cos$ and $\sin$, respectively, gives that the symmetric part of this matrix has trace
$-|z_j|^2\cos\mu_j$, and that its skew part is given by
\[
        -\frac{|z_j|^2}{2}\sin\mu_j\,J.
\]
Summing the two identities over \(j\) gives \eqref{eq:trace-skew-identities}.

Next, the traceless symmetric part of the contribution of the \(j\)-th coordinate has the form
\[
        \frac{|z_j|^2}{2}
        \begin{bmatrix}
        \cos(2\psi_j+\pi-\mu_j)&\sin(2\psi_j+\pi-\mu_j)\\
        \sin(2\psi_j+\pi-\mu_j)&-\cos(2\psi_j+\pi-\mu_j)
        \end{bmatrix}.
\]
Thus \eqref{eq:traceless-vector-identity} holds with
\(\eta_j=2\psi_j+\pi-\mu_j\).  Applying \Cref{lem:polygon} to \eqref{eq:traceless-vector-identity} gives the lower bound on \(a^2+b^2\) in \eqref{eq:traceless-lower-det-upper}.  Finally,
\[
        \det H=\left(\frac{\tr H}{2}\right)^2-a^2-b^2,
\]
so \eqref{eq:trace-skew-identities} together with the lower bound on \(a^2+b^2\) implies the determinant upper bound in \eqref{eq:traceless-lower-det-upper}.
\end{proof}

If destabilization is possible, the following proposition gives a lower bound on the largest phase of a destabilizing matrix $\Delta$; this lower bound is obtained from a compression of $\Delta$ onto the witness plane. Moreover, it also gives a necessary condition for when destabilization is possible in terms of the inequality \eqref{eq:witness-q-rho}, i.e., a witness.

\begin{proposition}[Two-dimensional witness bound]\label{prop:witness-bound}
Let
\[
        M=\diag(e^{\ii\mu_1},\ldots,e^{\ii\mu_n}),
        \qquad
        \pi>\mu_1\ge\cdots\ge\mu_n>-\pi/2,
        \qquad
        \mu_1-\mu_n<\pi,
\]
and let \(\Delta\in\R^{n\times n}\) satisfy \(\Delta+\Delta^T \succ 0\) and \(\det(I+M\Delta)=0\).  Choose \(z\ne0\) such that \(\Delta z=-M^{-1}z\), and define \(p_j,q,r,\rho\) from $z$ and the phases $\mu_j$ by \eqref{eq:p-q-r-rho-def}.  Then
\begin{equation}\label{eq:witness-q-rho}
        q>\rho
\end{equation}
and
\begin{equation}\label{eq:witness-phase-bound}
        \phi_1(\Delta)
        \ge
        \arctan\frac{|r|}{\sqrt{q^2-\rho^2}}.
\end{equation}
\end{proposition}

\begin{proof}
The proof has four steps.  First, we form the real witness plane.  Second, we compress \(\Delta\) to that plane.  Third, we compute the trace, skew-Hermitian part, and determinant of the compressed \(2\times2\) matrix.  Fourth, we apply the explicit \(2\times2\) phase formula.

Let
\[
        X:=\begin{bmatrix}\Rea z&\Ima z\end{bmatrix},
        \qquad
        y:=-M^{-1}z,
        \qquad
        Y:=\begin{bmatrix}\Rea y&\Ima y\end{bmatrix}.
\]
By \Cref{lem:witness-plane-independent}, \(X\) has full column rank.  Since \(\Delta\) is real and \(\Delta z=y\), we have that $\Delta X=Y$ and hence
\[
        A:=X^TY=X^T\Delta X,
\]
i.e., $A$ is a compression of $\Delta$ onto the witness plane.
By \Cref{lem:compression}, \(A\) is strictly accretive and
\begin{equation}\label{eq:phi-A-leq-phi-Delta}
        \phi_1(A)\le\phi_1(\Delta).
\end{equation}

For each coordinate,
\[
        y_j=-e^{-\ii\mu_j}z_j,
\]
and hence we can apply \Cref{lem:coordinate-identities} to \(X,Y\).  To this end, write \(H=\He(A)\) and \(\Sk(A)=sJ\) and note that from \eqref{eq:trace-skew-identities} and \eqref{eq:traceless-lower-det-upper},
\[
        \tr H=\|z\|^2q,
        \qquad
        \det H\le \frac{\|z\|^4}{4}(q^2-\rho^2).
\]
Since \(H \succ 0\), we have that \(\tr H>0\) and \(\det H>0\), and hence \(q>0\) and \(q^2-\rho^2>0\), which gives \(q>\rho\).

Finally, \Cref{lem:2by2-phase} and \eqref{eq:trace-skew-identities} give
\[
        \tan\phi_1(A)
        =\frac{|s|}{\sqrt{\det H}}
        \ge
        \frac{(\|z\|^2/2)|r|}{(\|z\|^2/2)\sqrt{q^2-\rho^2}}
        =\frac{|r|}{\sqrt{q^2-\rho^2}}.
\]
Since \(0<\phi_1(A)<\pi/2\), applying \(\arctan\) and then using \eqref{eq:phi-A-leq-phi-Delta} proves \eqref{eq:witness-phase-bound}.
\end{proof}

\subsection{A bound in terms of phases of \texorpdfstring{$M$}{M}}\label{sec:scalar-optimization}

The previous section reduces the matrix problem to the scalar quantities \(q,r,\rho\), where \eqref{eq:witness-q-rho} is a necessary condition for the possibility to destabilize, and in which case \eqref{eq:witness-phase-bound} gives a lower bound on the largest phase of the matrix needed for destabilization. This section relates the quantities to the phases of $M$, and in particular to the largest and second largest phases.

\begin{lemma}[Scalar obstruction when the two largest phases are too small]\label{lem:scalar-obstruction}
Let
\[
        \pi>\mu_1\ge\mu_2\ge\cdots\ge\mu_n>-\pi/2,
        \qquad
        \mu_1-\mu_n<\pi,
\]
let \(p_j\ge0\), \(\sum_jp_j=1\), and define \(q,r,\rho\) by \eqref{eq:p-q-r-rho-def}.  If
\[
        \mu_1+\mu_2\le\pi,
\]
then
\begin{equation}\label{eq:q-le-rho-obstruction}
        q\le\rho.
\end{equation}
\end{lemma}

The conclusion $q\le\rho$ is a violation of \eqref{eq:witness-q-rho}, and thus our witness tells us that there is no $\Delta$ that can destabilize $M$ when $\mu_1+\mu_2\le\pi$.

\begin{proof}[Proof of Lemma~\ref{lem:scalar-obstruction}]
If $\mu_1\le\pi/2$, then all phases lie in $(-\pi/2,\pi/2]$. Hence $\cos\mu_j\ge0$ for every $j$, so $q\le0\le\rho$.

It remains to consider the case $\mu_1>\pi/2$. Put $a:=\pi-\mu_1$, so $0<a<\pi/2$. Since $\mu_1-\mu_n<\pi$, we have $\mu_j>-a$ for every $j$. Moreover, $\mu_1+\mu_2\le\pi$ implies $\mu_j\le\mu_2\le a$ for every $j\ge2$. Therefore $\cos\mu_j\ge\cos a$ for every $j\ge2$, and
\[
\begin{aligned}
        q
        &=-p_1\cos\mu_1-\sum_{j=2}^n p_j\cos\mu_j \\
        &\le p_1\cos a-(1-p_1)\cos a
        =(2p_1-1)\cos a .
\end{aligned}
\]
If $p_1\le1/2$, this gives $q\le0\le\rho$. If $p_1>1/2$, then
\[
        q\le(2p_1-1)\cos a<2p_1-1\le\rho.
\]
Thus $q\le\rho$ in all cases.
\end{proof}

In the case where destabilization is possible, the following two lemmas derive relations that we then use to bound the right-hand side of \eqref{eq:witness-phase-bound} in terms of $\mu_1$ and $\mu_2$.

\begin{lemma}
\label{lem:rotated-estimate}
Assume
\[
        \pi>\mu_1\ge\mu_2\ge\cdots\ge\mu_n>-\pi/2,
        \qquad
        \mu_1-\mu_n<\pi,
        \qquad
        \mu_1+\mu_2>\pi.
\]
Let
\[
        s:=\frac{\mu_1+\mu_2}{2},
        \qquad
        d:=\frac{\mu_1-\mu_2}{2}.
\]
Then $0\le d<\pi-s<\pi/2$, and every choice of weights \(p_j\ge0\), \(\sum_jp_j=1\), satisfies
\begin{equation}\label{eq:rotated-estimate}
        q\sin s+r\cos s\le \rho\sin d,
\end{equation}
where \(q,r,\rho\) are defined by \eqref{eq:p-q-r-rho-def}. Moreover, if $q>\rho$, then $r>0$.
\end{lemma}

\begin{proof}
The assumptions imply $\mu_2>0$, and hence $0\le d<s$. Since $\mu_1=s+d<\pi$, we also have $d<\pi-s<\pi/2$.
Now,
\[
        q\sin s+r\cos s
        =\sum_{j=1}^n p_j\sin(\mu_j-s).
\]
For $j=1$, $\sin(\mu_1-s)=\sin d$. For every $j\ge2$, the ordering and the sectorial width condition give
\[
        d-\pi<\mu_j-s\le -d.
\]
The maximum of \(\sin t\) on \((d-\pi, -d]\) is \(\sin (-d) = -\sin d\). Therefore
\[
        q\sin s+r\cos s
        \le p_1\sin d-(1-p_1)\sin d
        =(2p_1-1)\sin d
        \le \rho\sin d,
\]
which proves \eqref{eq:rotated-estimate}.

It remains to prove the final assertion. Put $a:=\pi-\mu_1$. Then $0<a<\pi/2$. Since $\mu_1-\mu_n<\pi$, we have $\mu_j+a>0$ for every $j$, and since $\mu_j\le\mu_1$, we have $\mu_j+a\le\pi$. Thus $\sin(\mu_j+a)\ge0$ for every $j$. Taking the weighted sum gives
\[
        0\le\sum_{j=1}^n p_j\sin(\mu_j+a)
        =r\cos a+\sum_{j=1}^n p_j\cos\mu_j\sin a
        =r\cos a-q\sin a.
\]
If $q>\rho$, then $q>0$, and hence $r\ge q\tan a>0$.
\end{proof}

\begin{lemma}[Algebraic angle conversion]\label{lem:algebraic-angle}
Let \(0\le d<m<\pi/2\), \(q>\rho\ge0\), and \(r>0\).  If
\begin{equation}\label{eq:algebraic-hypothesis}
        r\cos m-q\sin m\ge -\rho\sin d,
\end{equation}
then
\begin{equation}\label{eq:algebraic-angle-conclusion}
        \arctan\frac{r}{\sqrt{q^2-\rho^2}}
        \ge
        \arccos\frac{\cos m}{\cos d}.
\end{equation}
\end{lemma}

\begin{proof}
The proof is mostly algebraic, and has little to do with the geometric properties and estimates. To highlight this, which also makes the proof easier to follow, put
\[
        A:=\sin m > 0,
        \qquad
        B:=\sin d \geq 0,
        \qquad
        C:=\cos m > 0,
\]
and note that $A > B$. With this notation, \eqref{eq:algebraic-hypothesis} reads $rC -q A \geq -\rho B$, which gives
\[
        q\le\frac{rC+\rho B}{A}.
\]
Since $q > 0$, squaring both sides and then subtracting $\rho^2$ on both sides gives
\[
        q^2-\rho^2
        \le
        \frac{(rC+\rho B)^2}{A^2}-\rho^2.
\]
Multiplying both sides by $(A^2 - B^2) > 0$ gives
\[
        (A^2-B^2)(q^2-\rho^2)\le (A^2-B^2)\left(\frac{(rC+\rho B)^2}{A^2}-\rho^2\right)  .
\]
Now, completing the square gives
\[
        C^2r^2 - (A^2-B^2)\left(\frac{(rC+\rho B)^2}{A^2}-\rho^2\right) = \frac{(BCr-(A^2-B^2)\rho)^2}{A^2}\ge0,
\]
which in turn gives
\[
        (A^2-B^2)(q^2-\rho^2)\le C^2r^2.
\]
Since
\[
        A^2-B^2=\sin^2m-\sin^2d=\cos^2d-\cos^2m,
\]
we obtain
\[
        \frac{q^2-\rho^2}{q^2-\rho^2+r^2}
        \le
        \frac{\cos^2m}{\cos^2d}.
\]
Taking square roots gives
\[
        \cos\left(\arctan\frac{r}{\sqrt{q^2-\rho^2}}\right)
        \le\frac{\cos m}{\cos d}.
\]
Both angles lie in \((0,\pi/2)\), so this is equivalent to \eqref{eq:algebraic-angle-conclusion}.
\end{proof}

Taking this together, we have the following proposition.

\begin{proposition}[Two-angle scalar bound]\label{prop:scalar-two-angle-bound}
Let
\[
        \pi>\mu_1\ge\mu_2\ge\cdots\ge\mu_n>-\pi/2,
        \qquad
        \mu_1-\mu_n<\pi,
\]
let \(p_j\ge0\), \(\sum_jp_j=1\), and define \(q,r,\rho\) by \eqref{eq:p-q-r-rho-def}.

\begin{enumerate}[label=(\alph*)]
\item If \(\mu_1+\mu_2\le\pi\), then \(q\le\rho\).

\item If \(\mu_1+\mu_2>\pi\) and \(q>\rho\), then $r>0$ and
\begin{equation}\label{eq:scalar-two-angle-bound}
        \arctan\frac{r}{\sqrt{q^2-\rho^2}}
        \ge
        \Theta,
\end{equation}
where \(0<\Theta<\pi/2\) is determined by
\begin{equation}\label{eq:Theta-mu-scalar}
        \cos\Theta=
        -\frac{\cos((\mu_1+\mu_2)/2)}
              {\cos((\mu_1-\mu_2)/2)}.
\end{equation}
\end{enumerate}
\end{proposition}

\begin{proof}
Part (a) is \Cref{lem:scalar-obstruction}.  For part (b), set
\[
        s=\frac{\mu_1+\mu_2}{2},
        \qquad
        d=\frac{\mu_1-\mu_2}{2},
        \qquad
        m=\pi-s.
\]
By \Cref{lem:rotated-estimate}, $0\le d<m<\pi/2$, $r>0$, and
\[
        q\sin s+r\cos s\le\rho\sin d.
\]
Since $\cos m=-\cos s$ and $\sin m=\sin s$, this is equivalent to
\[
        r\cos m-q\sin m\ge -\rho\sin d.
\]
Therefore \Cref{lem:algebraic-angle} gives
\[
        \arctan\frac{r}{\sqrt{q^2-\rho^2}}
        \ge
        \arccos\frac{\cos m}{\cos d}
        =
        \arccos\!\left(
        -\frac{\cos((\mu_1+\mu_2)/2)}
              {\cos((\mu_1-\mu_2)/2)}
        \right),
\]
which is exactly \eqref{eq:scalar-two-angle-bound}--\eqref{eq:Theta-mu-scalar}.
\end{proof}

\subsection{Proof of Theorem~\ref{thm:main}}\label{subsec:main-theorem}

\begin{proof}[Proof of Theorem~\ref{thm:main}]
The proof consists of a lower bound and a matching construction.  The lower bound uses the two-dimensional witness bound from \Cref{sec:forced-compression} and the scalar optimization from \Cref{sec:scalar-optimization}.  The construction is a single real rotation in the coordinate plane corresponding to the two largest phases of \(M\), cf.~the proof of \Cref{prop:complex_real_margin_same}.

Suppose first that a destabilizing strictly accretive real \(\Delta\) exists. By \Cref{prop:witness-bound}, the associated weights satisfy \(q>\rho\).  If \(\mu_1+\mu_2\le\pi\), this contradicts \Cref{prop:scalar-two-angle-bound}(a).  Therefore no such \(\Delta\) exists in this case.  This proves item (i).

Now assume \(\mu_1+\mu_2>\pi\).  Let \(\Delta\) be any destabilizing real strictly accretive perturbation.  Combining \Cref{prop:witness-bound} and \Cref{prop:scalar-two-angle-bound}(b) gives
\[
        \phi_1(\Delta)
        \ge
        \Theta,
\]
where
\[
        \cos\Theta
        =-
        \frac{\cos((\mu_1+\mu_2)/2)}
             {\cos((\mu_1-\mu_2)/2)}.
\]
Thus \(\theta_{\R}(M)\ge\Theta\).
It remains to show that this lower bound is attained. To this end, define
\[
        Q_\Theta:=
        \begin{bmatrix}
        \cos\Theta&-\sin\Theta\\
        \sin\Theta& \cos\Theta
        \end{bmatrix},
        \qquad
        \Delta_\Theta:=Q_\Theta\oplus I_{n-2},
\]
with the identity block omitted when \(n=2\).  Because \(\mu_1+\mu_2>\pi\), the assumptions imply $\mu_2>0$. Moreover, with $s=(\mu_1+\mu_2)/2$ and $d=(\mu_1-\mu_2)/2$, we have $\pi/2<s<\pi$ and $0\le d<\pi-s$, so $0<-\cos s/\cos d<1$. Hence $0<\Theta<\pi/2$. Therefore,
\[
        \Delta_\Theta+\Delta_\Theta^T=2(\cos\Theta)I_2\oplus 2I_{n-2} \succ 0,
\]
and the phases of \(\Delta_\Theta\) are \(\Theta, 0,\ldots,0, -\Theta\).  Thus, \(\phi_1(\Delta_\Theta)=\Theta\).
Finally, we note that the leading \(2\times2\) block of \(I+M\Delta_\Theta\) is singular, because
\begin{align*}
&\det\left(I_2+
        \diag(e^{\ii\mu_1},e^{\ii\mu_2})Q_\Theta\right) \\
&\quad=
        1+(e^{\ii\mu_1}+e^{\ii\mu_2})\cos\Theta
        +e^{\ii(\mu_1+\mu_2)} \\
&\quad=
        1- (e^{\ii\mu_1}+e^{\ii\mu_2})
        \frac{\cos((\mu_1+\mu_2)/2)}{\cos((\mu_1-\mu_2)/2)}
        +e^{\ii(\mu_1+\mu_2)} \\
&\quad=
        1-2e^{\ii(\mu_1+\mu_2)/2}\cos((\mu_1+\mu_2)/2)
        +e^{\ii(\mu_1+\mu_2)}=0.
\end{align*}
All remaining diagonal factors are \(1+e^{\ii\mu_j}\ne0\), since $-\pi/2<\mu_j<\pi$.  Hence \(\det(I+M\Delta_\Theta)=0\), whereby \(\theta_{\R}(M)\le\Theta\).  The lower and upper bounds coincide and prove \eqref{eq:main-formula}.
\end{proof}

\section{Conclusion}\label{sec:conclusion}
In this paper, we studied the real angular stability radius problem for sectorial matrices. In particular, we showed that the complex and the real radius can be equal, but that for diagonal and for complex symmetric matrices they are in general different. For such matrices, the radius -- which can be infinite if destabilization is not possible -- is determined by the two largest phases of the matrix rather than only by the largest one. In particular, in the upper-dominant case covered by \Cref{thm:main}, destabilization with a real accretive perturbation is only possible when these two phases satisfy $\phi_1(M) + \phi_2(M) >\pi$. The proof also identifies the mechanism behind this condition: it is forced by the result for matrices of size $2 \times 2$. More specifically, any real destabilizing perturbation gives rise to a two-dimensional real witness plane, obtained from the real and imaginary parts of a destabilizing vector. The compression of the destabilizing matrix to this plane retains enough phase information to give a sharp lower bound, and the bound is attained by a real rotation in the coordinate plane corresponding to the two largest phases.

The paper is a first investigation of the real angular stability radius, and there are several open and interesting questions. For example, it would be of interest to better understand the real angular stability radius for sectorial matrices that are neither diagonal nor complex symmetric, as well as to investigate the radius for matrices $M$ that are not necessarily sectorial. It would also be worthwhile to extend the present finite-dimensional matrix results to the system-theoretic phase setting, where phase concepts for linear and nonlinear dynamical systems have recently been developed; see \Cref{rem:phases_mimo_systems}.

\section*{Declaration of generative AI and AI-assisted technologies in the manuscript preparation process}

During the preparation of this work, the authors used ChatGPT 5.5 Pro in order to develop the main proof; it was developed over multiple queries, and in the queries the authors included drafts of incomplete attempts to prove the statement which were produced by the authors.%
\footnote{In particular, the proof in \Cref{app:2times2} was derived without the use of any AI technology, and it was supplied to the AI tool during the derivation of the main proof.}
The authors also used ChatGPT 5.5 Pro to assist with text refinement. After using this tool/service, the authors reviewed and edited the content as needed, and take full responsibility for the content of the published article.

\section*{Acknowledgments}

S.Z.K.'s work was supported in part by the National Science and Technology Council of Taiwan under grants 113-2222-E-110-002-MY3, 115-2218-E-007-003, 114-2622-8-110-003, and 115-2221-E-110-058-MY2. A.R.~acknowledge financial supported in part by the Knut and Alice Wallenberg foundation under Grant KAW 2018.0349, and in part by the Wallenberg AI, Autonomous Systems and Software Program (WASP) funded by the Knut and Alice Wallenberg Foundation.

\bibliographystyle{plain}
\bibliography{}

\begin{thebibliography}{10}

\bibitem{anderson1988hilbert}
Brian~D.O. Anderson and Michael Green.
\newblock Hilbert transform and gain/phase error bounds for rational functions.
\newblock {\em IEEE Transactions on Circuits and Systems}, 35(5):528--535,
  1988.

\bibitem{ArlinskiiPopov2003}
Yu.~M. Arlinskii and A.~B. Popov.
\newblock On sectorial matrices.
\newblock {\em Linear Algebra and its Applications}, 370:133--146, 2003.

\bibitem{BallantineJohnson1975}
C.~S. Ballantine and C.~R. Johnson.
\newblock Accretive matrix products.
\newblock {\em Linear and Multilinear Algebra}, 3(3):169--185, 1975.

\bibitem{BarOnJonckheere1990}
J.~R. Bar-On and E.~A. Jonckheere.
\newblock Phase margins for multivariable control systems.
\newblock {\em International Journal of Control}, 52(2):485--498, 1990.

\bibitem{BedraniKittanehSababheh2021Positivity}
Yassine Bedrani, Fuad Kittaneh, and Mohammed Sababheh.
\newblock From positive to accretive matrices.
\newblock {\em Positivity}, 25(4):1601--1629, 2021.

\bibitem{BedraniKittanehSababheh2021AFA}
Yassine Bedrani, Fuad Kittaneh, and Mohammed Sababheh.
\newblock On the weighted geometric mean of accretive matrices.
\newblock {\em Annals of Functional Analysis}, 12(2), 2021.

\bibitem{Byers1988}
Ralph Byers.
\newblock A bisection method for measuring the distance of a stable matrix to
  the unstable matrices.
\newblock {\em {SIAM} Journal on Scientific and Statistical Computing},
  9(5):875--881, 1988.

\bibitem{chen2020phase}
Chao Chen, Di~Zhao, Wei Chen, Sei~Zhen Khong, and Li~Qiu.
\newblock Phase of nonlinear systems.
\newblock {\em Preprint: arXiv:2012.00692}, 2020.

\bibitem{chen2025singular}
Chao Chen, Di~Zhao, and Sei~Zhen Khong.
\newblock The singular angle of nonlinear systems.
\newblock {\em Automatica}, 181:112515, 2025.

\bibitem{chen1998multivariable}
Jie Chen.
\newblock Multivariable gain-phase and sensitivity integral relations and
  design trade-offs.
\newblock {\em IEEE Transactions on Automatic Control}, 43(3):373--385, 1998.

\bibitem{ChenWangKhongQiu2019}
Wei Chen, Dan Wang, Sei~Zhen Khong, and Li~Qiu.
\newblock Phase analysis of {MIMO} {LTI} systems.
\newblock In {\em Proceedings of the 58th {IEEE} Conference on Decision and
  Control}, pages 6062--6067, 2019.

\bibitem{ChenWangKhongQiu2024}
Wei Chen, Dan Wang, Sei~Zhen Khong, and Li~Qiu.
\newblock A phase theory of multi-input multi-output linear time-invariant
  systems.
\newblock {\em {SIAM} Journal on Control and Optimization}, 62(2):1235--1260,
  2024.

\bibitem{DePrimaJohnson1974}
Charles~R. DePrima and Charles~R. Johnson.
\newblock The range of {$A^{-1}A^*$} in {$GL(n,\mathbb C)$}.
\newblock {\em Linear Algebra and its Applications}, 9:209--222, 1974.

\bibitem{Doyle1982}
J.~C. Doyle.
\newblock Analysis of feedback systems with structured uncertainties.
\newblock {\em {IEE} Proceedings D - Control Theory and Applications},
  129(6):242--250, 1982.

\bibitem{drury2015principal}
Stephen Drury.
\newblock Principal powers of matrices with positive definite real part.
\newblock {\em Linear and Multilinear Algebra}, 63(2):296--301, 2015.

\bibitem{DruryLin2014}
Stephen Drury and Minghua Lin.
\newblock Singular value inequalities for matrices with numerical ranges in a
  sector.
\newblock {\em Operators and Matrices}, 8(4):1143--1148, 2014.

\bibitem{FurtadoJohnson2001}
Susana Furtado and Charles~R. Johnson.
\newblock Spectral variation under congruence.
\newblock {\em Linear and Multilinear Algebra}, 49(3):243--259, 2001.

\bibitem{FurtadoJohnson2003}
Susana Furtado and Charles~R. Johnson.
\newblock Spectral variation under congruence for a nonsingular matrix with 0
  on the boundary of its field of values.
\newblock {\em Linear Algebra and its Applications}, 359:67--78, 2003.

\bibitem{GeorgeIkramov2005}
A.~George and Kh.~D. Ikramov.
\newblock On the properties of accretive-dissipative matrices.
\newblock {\em Mathematical Notes}, 77(6):767--776, 2005.

\bibitem{GreenLimebeer1995}
Michael Green and David J.~N. Limebeer.
\newblock {\em Linear Robust Control}.
\newblock Prentice Hall, Englewood Cliffs, NJ, 1995.

\bibitem{haddad1992there}
Wassim~M Haddad and DS~Bernstein.
\newblock Is there more to robust control theory than small gain?
\newblock In {\em Proceedings of the 1992 American Control Conference}, pages
  83--84, 1992.

\bibitem{HinrichsenPritchard1986a}
D.~Hinrichsen and A.~J. Pritchard.
\newblock Stability radii of linear systems.
\newblock {\em Systems \& Control Letters}, 7(1):1--10, 1986.

\bibitem{HinrichsenPritchard1986b}
D.~Hinrichsen and A.~J. Pritchard.
\newblock Stability radius for structured perturbations and the algebraic
  {Riccati} equation.
\newblock {\em Systems \& Control Letters}, 8(2):105--113, 1986.

\bibitem{HinrichsenPritchard1990}
D.~Hinrichsen and A.~J. Pritchard.
\newblock Real and complex stability radii: a survey.
\newblock In D.~Hinrichsen and B.~M{\aa}rtensson, editors, {\em Control of
  Uncertain Systems}, pages 119--162. Birkh{\"a}user, Boston, 1990.

\bibitem{HornSteinberg1959}
Alfred Horn and Robert Steinberg.
\newblock Eigenvalues of the unitary part of a matrix.
\newblock {\em Pacific Journal of Mathematics}, 9(2):541--550, 1959.

\bibitem{horn1994topics}
Roger~A. Horn and Charles~R. Johnson.
\newblock {\em Topics in matrix analysis}.
\newblock Cambridge University Press, New York, NY, 1994.

\bibitem{HornSergeichuk2006}
Roger~A. Horn and Vladimir~V. Sergeichuk.
\newblock Canonical forms for complex matrix congruence and {*}congruence.
\newblock {\em Linear Algebra and its Applications}, 416(2--3):1010--1032,
  2006.

\bibitem{JohnsonFurtado2001}
Charles~R. Johnson and Susana Furtado.
\newblock A generalization of {Sylvester}'s law of inertia.
\newblock {\em Linear Algebra and its Applications}, 338(1--3):287--290, 2001.

\bibitem{LiSze2014}
Chi-Kwong Li and Nung-Sing Sze.
\newblock Determinantal and eigenvalue inequalities for matrices with numerical
  ranges in a sector.
\newblock {\em Journal of Mathematical Analysis and Applications},
  410(1):487--491, 2014.

\bibitem{MaoChenQiu2022}
Xin Mao, Wei Chen, and Li~Qiu.
\newblock Phases of discrete-time {LTI} multivariable systems.
\newblock {\em Automatica}, 142:110311, 2022.

\bibitem{marshall2011inequalities}
Albert~W Marshall, Ingram Olkin, and Barry~C Arnold.
\newblock {\em Inequalities: Theory of Majorization and its Applications}.
\newblock Springer, New York, NY, 2nd edition, 2011.

\bibitem{owens1984numerical}
David~H Owens.
\newblock The numerical range: a tool for robust stability studies?
\newblock {\em Systems \& Control Letters}, 5(3):153--158, 1984.

\bibitem{PackardDoyle1993}
A.~Packard and J.~Doyle.
\newblock The complex structured singular value.
\newblock {\em Automatica}, 29(1):71--109, 1993.

\bibitem{PostlethwaiteEdmundsMacFarlane1981}
I.~Postlethwaite, J.~M. Edmunds, and A.~G.~J. MacFarlane.
\newblock Principal gains and principal phases in the analysis of linear
  multivariable feedback systems.
\newblock {\em {IEEE} Transactions on Automatic Control}, 26(1):32--46, 1981.

\bibitem{QiuBernhardssonRantzerDavisonYoungDoyle1995}
Li~Qiu, Bo~Bernhardsson, Anders Rantzer, E.~J. Davison, P.~M. Young, and J.~C.
  Doyle.
\newblock A formula for computation of the real stability radius.
\newblock {\em Automatica}, 31(6):879--890, 1995.

\bibitem{QiuDavison1989}
Li~Qiu and E.~J. Davison.
\newblock A simple procedure for the exact stability robustness computation of
  polynomials with affine coefficient perturbations.
\newblock {\em Systems \& Control Letters}, 13(5):413--420, 1989.

\bibitem{QiuDavison1991}
Li~Qiu and E.~J. Davison.
\newblock The stability robustness determination of state space models with
  real unstructured perturbations.
\newblock {\em Mathematics of Control, Signals, and Systems}, 4(3):247--267,
  1991.

\bibitem{QiuDavison1992}
Li~Qiu and E.~J. Davison.
\newblock Bounds on the real stability radius.
\newblock In M.~Mansour, S.~Balemi, and W.~Truol, editors, {\em Robustness of
  Dynamic Systems with Parameter Uncertainties}, pages 139--145.
  Birkh{\"a}user, Basel, 1992.

\bibitem{RinghQiu2022}
Axel Ringh and Li~Qiu.
\newblock Finsler geometries on strictly accretive matrices.
\newblock {\em Linear and Multilinear Algebra}, 70(20):6753--6771, 2022.

\bibitem{VanLoan1985}
Charles Van~Loan.
\newblock How near is a stable matrix to an unstable matrix?
\newblock In {\em Linear Algebra and Its Role in Systems Theory}, volume~47 of
  {\em Contemporary Mathematics}, pages 465--478. American Mathematical
  Society, 1985.

\bibitem{WangChenKhongQiu2020}
Dan Wang, Wei Chen, Sei~Zhen Khong, and Li~Qiu.
\newblock On the phases of a complex matrix.
\newblock {\em Linear Algebra and its Applications}, 593:152--179, 2020.

\bibitem{WangMaoChenQiu2023}
Dan Wang, Xin Mao, Wei Chen, and Li~Qiu.
\newblock On the phases of a semi-sectorial matrix and the essential phase of a
  {Laplacian}.
\newblock {\em Linear Algebra and its Applications}, 676:441--458, 2023.

\bibitem{yang2025small}
Xiaokan Yang, Wei Chen, and Li~Qiu.
\newblock The small phase condition is necessary for symmetric systems.
\newblock {\em Preprint: arXiv:2507.06617}, 2025.

\bibitem{zhang2025matrix}
Ding Zhang, Axel Ringh, and Li~Qiu.
\newblock Matrix completion and decomposition in phase-bounded cones.
\newblock {\em SIAM Journal on Matrix Analysis and Applications},
  46(2):837--857, 2025.

\bibitem{zhang2025phantom}
Ding Zhang, Xiaokan Yang, Axel Ringh, and Li~Qiu.
\newblock The phantom of {D}avis-{W}ielandt shell: A unified framework for
  graphical stability analysis of mimo lti systems.
\newblock {\em Preprint: arXiv:2507.19918}, 2025.

\bibitem{Zhang2015}
Fuzhen Zhang.
\newblock A matrix decomposition and its applications.
\newblock {\em Linear and Multilinear Algebra}, 63(10):2033--2042, 2015.

\bibitem{zhao2022low}
Di~Zhao, Axel Ringh, Li~Qiu, and Sei~Zhen Khong.
\newblock Low phase-rank approximation.
\newblock {\em Linear Algebra and its Applications}, 639:177--204, 2022.

\bibitem{ZhouDoyleGlover1996}
Kemin Zhou, John~C. Doyle, and Keith Glover.
\newblock {\em Robust and Optimal Control}.
\newblock Prentice Hall, Upper Saddle River, NJ, 1996.

\end{thebibliography}

\appendix

\crefalias{section}{appendix}

\section{Alternative proof for the \texorpdfstring{$2 \times 2$}{2 x 2} case}\label{app:2times2}

For this proof, we need the notion of majorization; we only introduce the basic notions and we refer to, e.g., \cite{marshall2011inequalities} for a comprehensive treatment of the subject.%
\footnote{Note that the phases of a sectorial matrix and the phases of the eigenvalues of the same matrix obey majorization inequalities \cite[Thm.~1]{FurtadoJohnson2001}, \cite[Lem.~2.3]{WangChenKhongQiu2020} (cf.~\cite[Lem.~9]{HornSteinberg1959}); these are analogous to majorization inequalities between singular values and magnitudes of eigenvalues, see, e.g., \cite[Thm~9.E.1]{marshall2011inequalities}.}
To this end, let $x, y \in \mR^n$; by $x^{\downarrow}$ we denote the vector obtained by sorting $x$ in a nonincreasing order, i.e., $x^{\downarrow} = [x_k^\downarrow]_{k=1}^n$ where $x_1^\downarrow \geq x_2^\downarrow \geq \ldots \geq x_n^{\downarrow}$. We say that $x$ is weakly majorized by $y$, and write $x \prec_w y$, if
\[
\sum_{k = 1}^{\ell} x_k^\downarrow \leq \sum_{k = 1}^{\ell} y_k^\downarrow, \quad \ell = 1, \ldots, n.
\]
Moreover, we say $x$ is majorized by $y$, denoted by $x \prec_m y$, if $x \prec_w y$ and $\sum_{k = 1}^{n} x_k^\downarrow = \sum_{k = 1}^{n} y_k^\downarrow$.

Now, consider a diagonal unitary matrix
\[
        M=
        \begin{bmatrix}
        e^{\ii\alpha}&0\\
        0&e^{\ii\beta}
        \end{bmatrix},
\]
with $\alpha\in(\pi/2,\pi)$ and $\beta\in[\pi-\alpha,\alpha]$.
Define
\[
        \widehat\theta(M)
        =\inf\{\phi_1(Q) \; : \;Q\in\R^{2\times2},\ Q\text{ is orthogonal},\
        Q+Q^T \succ 0\text{ and }\det(I+MQ)=0\},
\]
i.e., the problem in \eqref{eq:strict-radius} but with the extra constraint that $Q$ must be orthogonal. All strictly accretive orthogonal matrices $Q$ can be parametrized as
\[
        Q=\begin{bmatrix}
        \cos\theta&-\sin\theta\\
        \sin\theta&\cos\theta
        \end{bmatrix},
\]
where $\theta \in (-\pi/2, \pi/2)$.
Since $Q$ is invertible, we have that $\det(I+MQ)=0$ if and only if $\det(Q^{-1}+M)=0$. A direct calculation now gives that
\[
         \det(Q^{-1}+M)
        =1+(e^{\ii\alpha}+e^{\ii\beta})\cos\theta+e^{\ii(\alpha+\beta)}=0.
\]
Therefore,
\[        \cos\widehat\theta(M)
        =-\frac{1+\cos(\alpha+\beta)}{\cos\alpha+\cos\beta}
        =-\frac{\sin(\alpha+\beta)}{\sin\alpha+\sin\beta}
        =-\frac{\cos\frac{\alpha+\beta}{2}}{\cos\frac{\alpha-\beta}{2}},
\]
and hence
\begin{equation}\label{eq:thetahat}
        \widehat\theta(M)
        =\arccos\!\left(-\frac{\cos\frac{\alpha+\beta}{2}}
                              {\cos\frac{\alpha-\beta}{2}}\right).
\end{equation}
It remains to show that this is indeed the real angular stability radius, i.e., that the restriction to orthogonal matrices does not make the radius smaller.

To this end, we need the following two lemmas. They can both be inferred from \cite[Lem.~9]{HornSteinberg1959}.

\begin{lemma}\label{lem:angle_eig_angle_sectorial}
Let $A\in\C^{n\times n}$ be a nonsingular matrix and let $A=QP$ be its polar decomposition.  If $Q$ is sectorial, then
\[
        \angle\lambda(A) \prec_m \phi(Q).
\]
\end{lemma}

\begin{lemma}\label{lem:angles_polar_angles_sectorial}
Let $A\in\C^{n\times n}$ be a sectorial matrix with phases in $(\theta,\theta+\pi)$, where $\theta\in(-\pi,\pi]$, and let $A=QP$ be its polar decomposition.  It holds that
\[
        \phi(Q) \prec_m \phi(A).
\]
\end{lemma}

\begin{theorem}
Consider a diagonal unitary matrix $M\in\C^{2\times2}$ with
\[
        \phi_1(M)\in\left(\frac{\pi}{2},\pi\right),
        \qquad
        \phi_2(M)\in[\pi-\phi_1(M),\phi_1(M)].
\]
Then
\[
        \det(I+M\Delta)\ne0
        \quad\text{for all }\Delta\in\SR[\alpha]
\]
if and only if $\alpha<\widehat\theta(M)$.
\end{theorem}

\begin{proof}
The necessity part is straightforward; hence we only need to prove the sufficient part.

Consider $\Delta\in\SR[\alpha]$, where $\alpha<\widehat\theta(M)$.  Using polar decomposition, $\Delta$ can be expressed as $\Delta=QP$, where $P$ is positive definite and $Q$ is orthogonal.  According to \Cref{lem:angles_polar_angles_sectorial},
\[
        \phi(Q) \prec_m \phi(\Delta).
\]
Therefore $\phi_1(Q) \le \phi_1(\Delta)<\widehat\theta(M)$, which implies
\[
        \det(I+MQ)\ne0.
\]
We then show that $MQ$ is sectorial by contradiction.  Write $MQ$ in the matrix form
\[      MQ=
        \begin{bmatrix}
        e^{\ii\alpha}\cos\theta&-e^{\ii\alpha}\sin\theta\\
        e^{\ii\beta}\sin\theta&e^{\ii\beta}\cos\theta
        \end{bmatrix}.
\]
Since $MQ$ is unitary, every eigenvalue $\lambda$ of $MQ$ satisfies $|\lambda|=1$.  Suppose $MQ$ is not sectorial.  Then
\[
        \lambda_1(MQ)=-\lambda_2(MQ).
\]
It follows that
\begin{equation}\label{eq:trace-zero}
        \tr(MQ)=(e^{\ii\alpha}+e^{\ii\beta})\cos\theta
        =\lambda_1(MQ)+\lambda_2(MQ)=0.
\end{equation}
The condition under which equality \eqref{eq:trace-zero} holds is either $e^{\ii\alpha}+e^{\ii\beta}=0$ or $\cos\theta=0$.  However, $e^{\ii\alpha}+e^{\ii\beta}=0$ contradicts the assumption that $M$ is sectorial, and $\cos\theta=0$ contradicts the fact that $\theta<\widehat\theta(M)<\pi/2$.

Since $MQ$ is sectorial, by applying \Cref{lem:angle_eig_angle_sectorial}, we get that
\[
        \angle\lambda(M\Delta) = \angle\lambda(MQP) \prec_m \phi(MQ).
\]
Therefore
\[
        \angle\lambda(M\Delta)\le\phi_1(MQ) = \max \{\angle \lambda_1(MQ), \angle \lambda_2(MQ) \} <\pi,
\]
where the equality follows since the phases of a sectorial normal matrix are the phases of the eigenvalues \cite[p.~158]{WangChenKhongQiu2020}, and where the last inequality follows since $\det(I+MQ)\ne0$. Thus $\det(I+M\Delta)\ne0$ is established.
\end{proof}

\end{document}